\documentclass[12pt]{amsart}

\usepackage{amssymb}
\usepackage{amsthm}
\usepackage{amsmath}
\usepackage{fontenc}
\usepackage{inputenc}
\usepackage{enumitem}
\usepackage{hyperref}
\usepackage{times}
\usepackage{graphicx}
\usepackage{multirow}
\usepackage{xcolor}
\usepackage{colortbl}
\usepackage{color}
\usepackage{array}
\usepackage{wasysym}
\usepackage{tikz}
\usetikzlibrary{scopes,arrows,decorations.pathmorphing,backgrounds,positioning,fit,petri,shapes,calc}

\makeatletter
\@namedef{subjclassname@2010}{%

\textup{2010} Mathematics Subject Classification}

\makeatother

\def\MR#1{\href{http://www.ams.org/mathscinet-getitem?mr=#1}{MR#1}}
\def\real{\hbox{\rm\setbox1=\hbox{I}\copy1\kern-.45\wd1 R}}
\def\natural{\hbox{\rm\setbox1=\hbox{I}\copy1\kern-.45\wd1 N}}

\newtheorem{theorem}{Theorem}[section] 
\newtheorem{lemma}[theorem]{Lemma}     
\newtheorem{corollary}[theorem]{Corollary}

\newtheorem{property}[theorem]{Property}

\title[Rigidity Sequences for PRWM Transformations]{Rigidity Sequences of Power Rationally Weakly Mixing Transformations}

\author[T. M. Adams]{Terrence M. Adams}

\address{U.S. Government\\9800 Savage Rd.\\Ft. Meade, MD 20755}

\email{terry@ganita.org}

\date{\today}

\begin{document}

\begin{abstract} 
We prove that a class of infinite measure preserving transformations, 
satisfying a "strong" weak mixing condition, generates all rigidity sequences 
of all conservative ergodic invertible measure preserving transformations 
defined on a Lebesgue $\sigma$-finite measure space. 
\end{abstract}


\maketitle

\section{Introduction}
In the finite measure preservance setting, it is known that the 
weak mixing condition has many equivalent formulations. 
In the infinite measure preservance setting, many of these 
formulations lead to different families 
of infinite measure preserving transformations. 
For a general account of weak mixing conditions 
of infinite measure preserving transformations, 
please see \cite{Aa77,Aar97,Aar13,BayYan14,Dai14,DGMS}. 
Figure \ref{wm-tree} displays several distinct weak mixing 
conditions for infinite measure preserving transformations. 
The stronger weak mixing conditions appear higher in the diagram. 

\begin{figure}
\label{wm-tree}
\begin{tikzpicture}

   {[line width=2pt]
   {[black] \draw (3,0 + .7) -- (3,0 + .7) node[near end,above] {Conservative Ergodic};
   }}
   {[line width=2pt]
   {[black] \draw (3,0 + 1.3) -- (3,0 + 1.7) node[near end,above] {Totally Ergodic};
   }}
   {[line width=2pt]
   {[black] \draw (3,0 + 2.3) -- (3,0 + 2.7) node[near end,above] {Spectral Weak Mixing};
   }}
   {[line width=2pt]
   {[black] \draw (3,0 + 3.3) -- (3,0 + 3.7) node[near end,above] {Ergodic with Isometric Coefficients};
   }}
    {[line width=2pt]
   {[black] \draw (3,0 + 4.3) -- (3,0 + 4.7) node[near end,above] {Double Ergodicity};
   }}
   {[line width=2pt]
   {[black] \draw (3,0 + 5.3) -- (3,0 + 5.7) node[near end,above] {Ergodic Index $2$};
   }}
   {[line width=2pt]
   {[black] \draw (3,0 + 6.3) -- (3,0 + 6.7) node[near end,above] {Ergodic Index $k$};
   }}
   {[line width=2pt]
   {[black] \draw (3,0 + 7.3) -- (3,0 + 7.7) node[near end,above] {$\infty$- Ergodic Index};
   }}
   {[line width=2pt]
   {[black] \draw (3,0 + 8.3) -- (3,0 + 8.7) node[near end,above] {Power Weak Mixing};
   }}
   {[line width=2pt]
   {[black] \draw (3,0 + 9.3) -- (3,0 + 9.7) node[near end,above] {{\bf Power Rational Weak Mixing}};
   }}


\draw[-] (3,5.3) .. controls +(right:1cm) and +(down:1cm) .. (7,7.0) node[above](myarrow1){Rational Weak Mixing};
\draw[-] (7,7.7) .. controls +(up:1cm) and +(right:1cm) .. (3,9.7) node[above](myarrow2){};

\node at (9,3.5)[cloud, draw,cloud puffs=40,cloud puff arc=120, aspect=1, inner sep=-10pt, outer sep=0pt] {
\begin{minipage}{4.5cm}
Other weak mixing:
\begin{itemize}[leftmargin=*,topsep=0pt,itemsep=0pt]
\setlength{\parskip}{0pt}%
\item Power Doubly Ergodic
\item Power Subsequence Rational Weak Mixing
\item $R$-set Weak Mixing 
\end{itemize}
\end{minipage}
};

\end{tikzpicture}
\caption{Weak Mixing with Infinite Measure}
\end{figure}
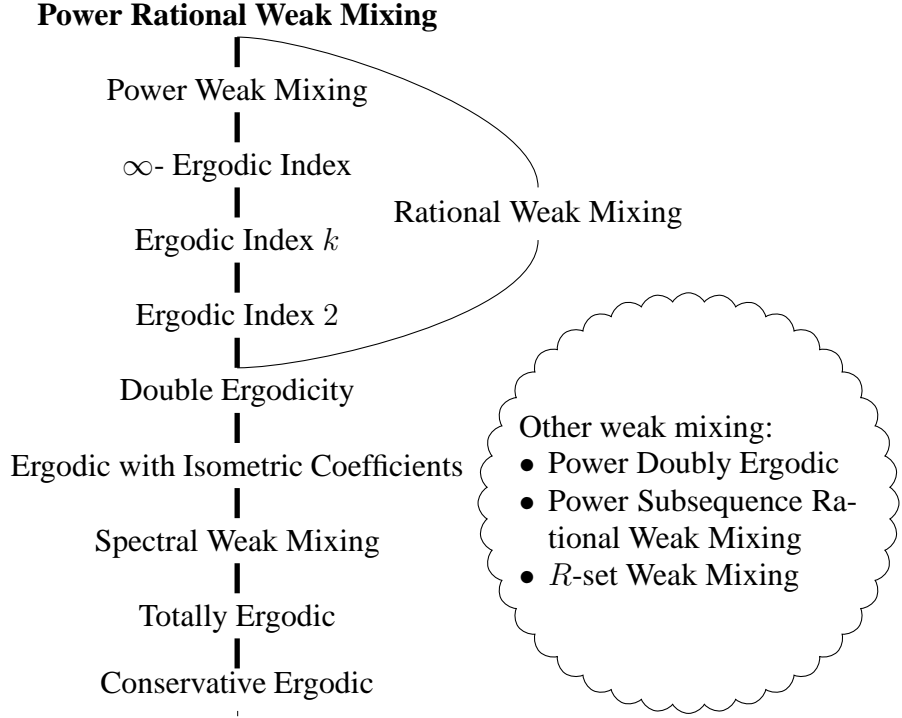
Most of the properties given in Figure \ref{wm-tree} were defined 
previously by multiple authors. Many interesting results have been derived. 
In the finite measure preserving case, it was proven that 
the collection of weak mixing transformations generates all rigidity 
sequences for all ergodic transformations defined on a Lebesgue space. 
See \cite{Towerplex1} and \cite{FayTho14} for details. 
For recent research on rigidity sequences 
in the $\sigma$-finite measure preserving case, 
see \cite{BayYan14,BdJLR,EisGri,Roy09,Roy12,Towerplex1}. 
Our primary goal is to give a class of infinite measure preserving, 
weak mixing transformations that generate all rigidity sequences of all 
ergodic finite measure preserving transformations. 
It was established in \cite{Roy09} and \cite{Roy12} 
that any rigidity sequence 
of a conservative ergodic infinite measure preserving transformation 
occurs as a rigidity sequence of a probability preserving weak mixing 
transformation. 
Thus, the class of infinite measure preserving transformations 
given here will generate all rigidity sequences for all 
conservative ergodic $\sigma$-finite measure preserving transformations. 
We find this more interesting, if we are able to restrict the class 
to a collection that satisfies a strong form of weak mixing. 

\begin{property}[Rational Weak Mixing] 
For any set $F\subset X$ of finite positive measure, 
define the intrinsic weight sequence 
of $F$, $u_k(F)$ and its accumulation by 
\begin{eqnarray}
u_k(F) = \frac{\mu(F\cap T^kF)}{\mu(F)^2} &\mbox{and}& a_n(F) = \sum_{k=0}^{n-1} u_k(F\cap T^kF) .
\end{eqnarray}
A $\sigma$-finite measure preserving transformation is rationally weakly mixing, 
if it is conservative ergodic and there exists a set $F$ of finite positive measure 
such that for all measurable sets $A,B \subset F$, 
\[
\lim_{n\to \infty} \frac{1}{a_n(F)} \sum_{k=0}^{n-1} | \mu(A\cap T^kB) - \mu(A)\mu(B)u_k(F) | = 0.
\]
\end{property}
Rational weak mixing was first introduced by Aaronson \cite{Aa77} 
as a counterpart to mixing on a sequence of density one 
in the finite measure preserving case. 
It is known that rational weak mixing implies double ergodicity 
which implies weak rational ergodicity \cite{Aa77}, 
and subsequently, implies spectral weak mixing. 
Rational weak mixing does not imply ergodic index 2, in general. 

In section \ref{main1}, 
we prove the following result. 
\begin{theorem}
\label{infinite_towerplex}
Let $(X,\mathcal{B},\mu)$ be a Lebesgue probability space. 
Suppose $R$ is an invertible 
ergodic $\mu$-preserving transformation on $(X,\mathcal{B},\mu)$ with 
a rigidity sequence $\rho_n \in \natural$ for $n\in \natural$. 
There exists an invertible infinite measure preserving transformation $T$ that is both 
rigid on $\rho_n$, and is rationally weakly mixing. 
\end{theorem}
\noindent 
We use the tower multiplexing technique 
given in \cite{Towerplex1}. In that paper, 
a rigid weakly mixing transformation 
is produced from multiplexing an ergodic rigid transformation 
with a weakly mixing transformation. 
All transformations were finite measure preserving. 
Here we wish to tower multiplex a finite measure preserving 
transformation with an infinite measure preserving 
transformation. Let $R$ be any finite measure preserving, 
ergodic transformation with rigidity sequence $\rho_n$. 
Since it was shown in \cite{Towerplex1} and \cite{FayTho14} that any rigid sequence 
of an ergodic finite measure preserving transformation may be realized 
by a finite measure preserving weak mixing transformation, then 
it is sufficient to assume the starter transformation $R$ is weak mixing, 
and rigid on $\rho_n$. 
The infinite measure preserving transformation $S$ will be akin to the map $S(x) = x + 1$ 
defined on $[0,\infty)$. The map $S$ is invertible, but it's not ergodic, 
which is not required for this construction. 
We will produce an infinite measure preserving transformation $T$ 
by multiplexing $R$ with $S$: 
\[
T = \mbox{Mux}\Big( \mbox{rigid weak mixing }R , \mbox{infinite measure preserving }S \Big) .
\] 

To strenghten our results, we introduce the notion 
of power rational weak mixing. 
An invertible infinite measure preserving transformation 
is {\it power rationally weakly mixing}, 
if given $\ell \in \natural$ and nonzero 
integers $k_1,k_2,\ldots ,k_{\ell}$, 
the Cartesian product 
\[
T^{k_1}\times T^{k_2}\times \ldots \times T^{k_{\ell}}
\]
is rationally weakly mixing. 
In this paper, assume all transformations are invertible 
and preserve a $\sigma$-finite measure defined on a Lebesgue space. 
Aaronson previously introduced the notion 
of power subsequence rational weak mixing \cite{Aar13}, 
and power weak mixing was introduced in \cite{DGMS}. 
Note, power weak mixing is defined as a transformation 
where all finite Cartesian products of nonzero powers 
are ergodic. Thus, power rational weak mixing 
implies power weak mixing. 
In the final section, we extend 
Theorem \ref{infinite_towerplex} to show that 
the class of power rational weak mixing, 
infinite measure preserving transformations 
generates all rigidity sequences 
for all finite measure preserving ergodic transformations. 
\begin{theorem}
\label{infinite_towerplex2}
Let $(X,\mathcal{B},\mu)$ be a Lebesgue probability space. 
Suppose $R$ is an invertible ergodic $\mu$-preserving transformation 
on $(X,\mathcal{B},\mu)$ with a rigidity sequence 
$\rho_n \in \natural$ for $n\in \natural$. 
There exists an invertible infinite measure preserving transformation 
$T$ that is both rigid on $\rho_n$, 
and is {\it power} rationally weakly mixing. 
\end{theorem}

\noindent 
A result from \cite{Roy09}, together with our result, 
show that rigidity sequences of ergodic finite measure 
preserving transformations coincide with rigidity 
sequences of conservative ergodic infinite measure 
preserving transformations. 
Moreover, the following corollary generalizes 
the main results from \cite{Towerplex1} and \cite{FayTho14}. 
\begin{corollary}
\label{inf-to-inf}
Let $(X,\mathcal{B},\mu)$ and $(Y,\mathcal{A},\nu)$ 
be Lebesgue $\sigma$-finite measure spaces. 
The set of rigidity sequences generated by all invertible 
{\it power} rationally weakly mixing transformations 
defined on $(Y,\mathcal{A},\nu)$
is identical to the set of rigidity sequences generated 
by all invertible conservative ergodic measure preserving 
transformations on $(X,\mathcal{B},\mu)$. 
\end{corollary}
\begin{proof}
The case where $\mu(X)$ and $\nu(Y)$ are finite 
is handled in \cite{Towerplex1} and \cite{FayTho14}. 
Suppose $\mu(X)=\infty$ and $\nu(Y)<\infty$. 
It is proved in \cite{Roy09} that the set 
of rigidity sequences generated by all conservative 
ergodic measure preserving transformations 
on $(X,\mathcal{B},\mu)$ is contained in the set 
of rigidity sequences generated by all invertible 
weak mixing transformations defined on $(Y,\mathcal{A},\nu)$. 
Theorem \ref{infinite_towerplex2} shows these sets are equal. 
Likewise, the case where $\mu(X)<\infty$ and $\nu(Y)=\infty$ 
follows from Theorem \ref{infinite_towerplex2} and \cite{Roy09}.  
Suppose both $\mu(X)$ and $\nu(Y)$ are infinite, and 
$R$ is infinite measure preserving and conservative ergodic 
on $(X,\mathcal{B},\mu)$.  By \cite{Roy09}, 
there exists a Poisson suspension $R^*$ such that 
$R^*$ is probability preserving, weak mixing and rigid on $\rho_n$. 
By Theorem \ref{infinite_towerplex2}, there exists an invertible 
infinite measure preserving power rationally weakly mixing 
$T$ that is rigid on $\rho_n$. 
\end{proof}

Note, recently, B. Fayad and A. Kanigowski  \cite{FayKan14} were able to construct a rigidity sequence 
for a finite measure preserving weak mixing transformation that is not rigid 
for any irrational rotation. This proves that the class of rigidity sequences for finite measure 
preserving weak mixing transformations is strictly larger than the class of rigidity sequences 
for finite measure preserving discrete spectrum transformations. 

Also, recently, R. Bayless and K. Yancey \cite{BayYan14} have given 
many explicit examples of infinite measure preserving transformations 
that are rigid and also satisfy a variety of weak mixing conditions 
(i.e. spectral weak mixing, rational ergodicity, ergodic Cartesian square). 

\section{Infinite Towerplex Constructions}
The towerplex method was first defined in section 2 of \cite{Towerplex1}.  
The use case here is simpler, since the only role of $S$ is to supply $T$ with infinite measure. 
There are a few main parameters that determine the final transformation. 
In \cite{Towerplex1}, 
two sequences $r_n$ and $s_n$ are defined such that $r_n$ represents the proportion of mass 
switching from the $R$-tower to the $S$-tower. Similarly, $s_n$ represents the proportion of mass 
switching from the $S$-tower to the $R$-tower. Using the notation from \cite{Towerplex1}, 
then the following values could be used to produce our desired transformation: 
\begin{eqnarray}
r_n = 0 &\mbox{and}& s_n = \frac{1}{n} . 
\end{eqnarray} 
Thus, for the constructions in this paper, we do not wish to transfer mass from the $R$-tower 
to the $S$-tower, and we wish to transfer measure ${1} / {n}$ from the $S$-tower 
to the $R$-tower at stage $n$.  It will be simpler to define $S_n:Y_n\to Y_n$ such that 
\begin{eqnarray}
\mu (Y_n) = \frac{1}{n} &\mbox{and}& s_n = 1 .  
\end{eqnarray} 
The transformation $S_n:Y_n\to Y_n$ will 
be a cycle on $h_n$ intervals, each with length ${1} / {nh_n}$ for some $h_n \in \natural$. 
The sequence $h_n$ will correspond to heights of Rohklin towers for the transformation $R_n$. 
At stage $n$ in the construction, a finite measure preserving, weak mixing transformation 
$R_n:X_n\to X_n$ will be defined to be isomorphic to $R$.  The set $X_n$ will be specified 
inductively. Finally, we will specify a sequence of refining, generating partitions $P_n$. 
The assembled transformation 
$T:\bigcup_{n=1}^{\infty}X_n \to \bigcup_{n=1}^{\infty}X_n$ will be invertible and 
ultimately (power) rationally weakly mixing with respect to $\mu$. 

\subsection{Towerplex Chain}
Suppose $h_n \in \natural$ and $\epsilon_n >0$ are such that 

\noindent 
$\sum_{n=1}^{\infty}{1}/{h_n} < \infty$ and $\sum_{n=1}^{\infty}\epsilon_n < \infty$. 
Initialize $R_1 = R$ on $X=X_1$ (ex. $X_1 = [0,1)$).  
Let $Y_1=[1,2)$ and define $S_1(x) = x + {1} / {h_1}$ on $[1,1 + {(h_1 - 1)}/{h_1})$ 
and $S(x)=x - {(h_1-1)} / {h_1}$ on $[1 + {(h_1-1)}/{h_1},2)$. 

Let $I_1, RI_1, \ldots ,R^{h_1-1}I_1$ be a Rohklin tower of height $h_1$ such that 
$\mu(E_1) < \epsilon_1$ where $E_1 = X_1\setminus \bigcup_{k=0}^{h_1-1}R_1 ^kI_1$. 
Let $X_2 = X_1 \cup Y_1$, and $d\in \real$ be such that 
\[
\frac{\mu(E_1) + d }{\mu(X_1)+\mu(Y_1)-d} = \frac{\mu(E_1) + d }{2 - d} = \frac{\mu(E_1)}{\mu(X_1)} .
\]
Let $J_1=[1, 1 + {1} / {h_1})$ be the base of $S_1$. 
Let $I_1^*$ be a subset of $J_1$ with measure ${ \vert d \vert } / {h_1}$. 
Let $X_1^{\prime} = E_1 \cup \bigcup_{k=0}^{h_1-1}R^k I_1^*$. 
Thus, $I_1\cup J_1 \setminus I_1^*, R_1 I_1 \cup S_1 (J_1 \setminus I_1^*), \ldots ,
R_1^{h_1-1} I_1 \cup S_1^{h_1-1} (J_1 \setminus I_1^*)$ are disjoint sets with equal measure. 
These sets together with $X_1^{\prime}$ may form a rescaled Rohklin tower for $R$. 
Now, we specify how to define $R_2$ consistently. 
Define 
\[
\mathcal{C}_1=\{ \bigcap_{k=0}^{h_1-1} R_1^{-k} p_k : p_k\in P_1, p_k\subset R_1^{k}I_1 \} . 
\]
The collection $\mathcal{C}_1$ generates a partition on $I_n$.  Let $P_{1}^{\prime}$ be the smallest 
partition generated by the collection:
\[
\bigvee_{k=0}^{h_1-1} \{ R_1^{k} p : p\in \mathcal{C}_1 \} \vee \{ S_1^k (J_1\setminus I_1^*) \} . 
\] 

Define $\tau_1:X_1^{\prime}\to E_1$ as a measure preserving map between 
normalized spaces 
$(X_1^{\prime},\mathbb{B} \cap X_1^{\prime},\frac{\mu}{\mu (X_1^{\prime})})$ and 
$(E_1,\mathbb{B} \cap E_1,\frac{\mu}{\mu (E_1)})$. 
Extend $\tau_1$ to the new tower base, 
$$\tau_1:I_1 \cup J_1\setminus I_1^*\to I_1$$
such that $\tau_1$ preserves normalized measure between 
$$\frac{\mu}{\mu (I_1 \cup J_1\setminus I_1^*)}\mbox{ and }\frac{\mu}{\mu (I_1)}.$$
Define $\tau_1$ on the remainder of the tower consistently 
such that 
\begin{eqnarray*} 
\tau_1(x)= 
\left\{\begin{array}{ll}
R_1^{i}\circ \tau_1 \circ R_1^{-i}(x) & \mbox{if $x\in R_1^i(I_1)$ for $0\leq i<h_1$} \\ 
R_1^{i}\circ \tau_1 \circ S_1^{-i}(x) & \mbox{if $x\in S_1^{i}(J_1\setminus I_1^*)$ for $0\leq i<h_1$} . 
\end{array}
\right.
\end{eqnarray*}
Since $\tau_1$ is a contraction, we may require for all $p\in P_1$,
\[
\tau_1(p) \subset p .
\]
Define 
$R_2:X_2\to X_2$ as $R_2=\tau_1^{-1}\circ R_1\circ \tau_1$. Note 
\begin{eqnarray*} 
R_2(x)= 
\left\{\begin{array}{ll}
S_1(x) & \mbox{if $x\in S_1^{i}(J_1\setminus I_1^*)$ for $0\leq i<h_1-1$} \\ 
R_1(x) & \mbox{if $x\in R_1^{i}(I_1)$ for $0\leq i<h_1-1$} 
\end{array}
\right.
\end{eqnarray*}
Clearly, $R_2$ is isomorphic to $R_1$ and $R$. 
Set $Y_2 = [2,2.5)$ and $S_2:Y_2\to Y_2$ by 
$S_2(x) = x + ({1} / {2h_2})$ for $x\in [2,2.5 - {1} / {2h_2})$ and 
$S_2(x) = x - ({h_2 - 1}) / {2h_2}$ for $x\in [2.5 - {1} / {2h_2}, 2.5)$. 
Let $b_1 = 2$ be the right endpoint of $Y_1$ and let 
$b_2 = 2.5$ be the right endpoint of $Y_2$. 

\subsection{General Multiplexing Operation}
Let $I_n, RI_n, \ldots ,R^{h_n-1}I_n$ be a Rohklin tower of height $h_n$ such that 
$\mu(E_n) < \epsilon_n$ where $E_n = X_n\setminus \bigcup_{k=0}^{h_n-1}R_n ^kI_n$. 
Suppose $b_n$ and $Y_n = [b_{n-1},b_n)$ have been defined.  
Let $X_{n+1} = X_n \cup Y_n$, and $d_n \in \real$ be such that 
\[
\frac{\mu(E_n) + d_n }{\mu(X_n)+\mu(Y_n)-d_n} = \frac{\mu(E_n)}{\mu(X_n)} . 
\]
Let $J_n=[0,{1} / {h_n})$ be the base of $S_n$. 
Let $I_n^*$ be a subset of $J_n$ with measure ${ \vert d_n \vert } / {h_n}$. 
Let $X_n^{\prime} = E_n \cup \bigcup_{k=0}^{h_n-1}R^k I_n^*$. 
The set $X_n^{\prime} \setminus E_n$ is the transfer set for stage $n$. 
Thus, $I_n\cup J_n \setminus I_n^*, R_n I_n \cup S_n (J_n \setminus I_n^*), \ldots ,
R_n^{h_n-1} I_n \cup S_n^{h_n-1} (J_n \setminus I_n^*)$ are disjoint sets with equal measure. 
These sets together with $X_n^{\prime}$ may form a rescaled Rohklin tower for $R$. 
Now, we specify how to define $R_{n+1}$ consistently. 
Define 
\[
\mathcal{C}_n=\{ \bigcap_{k=0}^{h_n-1} R_n^{-k} p_k : p_k\in P_{n-1}^{\prime} \vee P_{n}, p_k\subset R_n^{k}I_n \} . 
\]
The collection $\mathcal{C}_n$ generates a partition on $I_n$.  Let $P_{n}^{\prime}$ be the smallest 
partition generated by the collection:
\[
\bigvee_{k=0}^{h_n-1} \{ R_n^{k} p : p\in \mathcal{C}_n \} \vee \{ S_n^k (J_n\setminus I_n^*) \} . 
\] 

Define $\tau_n:X_n^{\prime}\to E_n$ as a measure preserving map between 
normalized spaces 
$(X_n^{\prime},\mathbb{B} \cap X_n^{\prime},\frac{\mu}{\mu (X_n^{\prime})})$ and 
$(E_n,\mathbb{B} \cap E_n,\frac{\mu}{\mu (E_n)})$. 
Extend $\tau_n$ to the new tower base, 
$$\tau_n:I_n \cup J_n\setminus I_n^*\to I_n$$
such that $\tau_n$ preserves normalized measure between 
$$\frac{\mu}{\mu (I_n \cup J_n\setminus I_n^*)}\mbox{ and }\frac{\mu}{\mu (I_n)}.$$
Define $\tau_n$ on the remainder of the tower consistently 
such that 
\begin{eqnarray*} 
\tau_n(x)= 
\left\{\begin{array}{ll}
R_n^{i}\circ \tau_n \circ R_n^{-i}(x) & \mbox{if $x\in R_n^i(I_n)$ for $0\leq i<h_n$} \\ 
R_n^{i}\circ \tau_n \circ S_n^{-i}(x) & \mbox{if $x\in S_n^{i}(J_n\setminus I_n^*)$ for $0\leq i<h_n$} . 
\end{array}
\right.
\end{eqnarray*}
Since $\tau_n$ is a contraction, we may require for all $p\in P_n^{\prime}$,
\[
\tau_n(p) \subset p .
\]
Define 
$R_{n+1}:X_{n+1}\to X_{n+1}$ as $R_{n+1}=\tau_n^{-1}\circ R_n\circ \tau_n$. Note 
\begin{eqnarray*} 
R_{n+1}(x)= 
\left\{\begin{array}{ll}
S_n(x) & \mbox{if $x\in S_n^{i}(J_n\setminus I_n^*)$ for $0\leq i<h_n-1$} \\ 
R_n(x) & \mbox{if $x\in R_n^{i}(I_n)$ for $0\leq i<h_n-1$} 
\end{array}
\right.
\end{eqnarray*}
Clearly, $R_{n+1}$ is isomorphic to $R_n$ and $R$.  
Set $b_{n+1} = b_n + {1}/{(n+1)}$, 
$Y_{n+1} = [b_n, b_{n+1})$ 
and transformation $S_{n+1}$ similar to the previous stages. 
Also, let 
\[
Q_n = \{ \tau_n (p) : p \in P_n^{\prime} \} . 
\]

\subsection{The Limiting Transformation}
Define the transformation 
$T_{n+1}:X_{n+1}\cup Y_{n+1}\to X_{n+1}\cup Y_{n+1}$ such that 
\begin{eqnarray*} 
T_{n+1}(x)= 
\left\{\begin{array}{ll}
R_{n+1}(x) & \mbox{if $x\in X_{n+1}$} \\ 
S_{n+1}(x) & \mbox{if $x\in Y_{n+1}$} 
\end{array}
\right.
\end{eqnarray*}
The set 
\[
D_n = \{ x\in X_{n+1} : T_{n+1} (x) \neq T_n (x) \}
\] 
is determined by the top levels 
of the Rokhlin towers, the residual and the transfer set. 
Note the transfer set has measure $d$. Since this set is used to adjust 
the size of the residuals between stages, it can be bounded below 
a constant multiple of $\epsilon_n$. Thus, there is a fixed constant $\kappa$, 
independent of $n$, such that 
$\mu (D_n) < \kappa (\epsilon_n + {1}/{h_n})$ . 
Since $\sum_{n=1}^{\infty} (\epsilon_n + {1}/{h_n}) < \infty$, 
$T(x) = \lim_{n\to \infty}T_n(x)$ exists almost everywhere, 
and preserves Lebesgue measure. 
Without loss of generality, we may assume $\kappa$ and $h_n$ 
are chosen such that for $n\in \natural$, 
\[
\mu (D_n) < \kappa \epsilon_n . 
\] 
Let $X^+ = \bigcup_{n=1}^{\infty} X_n$. 
Since 
\[
\mu(X^+) = \lim_{n\to \infty} (\mu(X_n) + \mu(Y_n)) = \infty,
\]
then $T$ is an invertible infinite measure preserving transformation. 
In the final section, we show there exist $h_n$ and $\epsilon_n$ 
such that $T:X^+ \to X^+$ is power rationally weakly mixing. 

\section{Isomorphism Chain Consistency}
Suppose $R$ is a weak mixing transformation on $(X,\mathcal{B},\mu)$ 
with rigidity sequence $\rho_n$.  
We will use the multiplexing procedure defined in the previous section 
to produce an invertible infinite measure preserving $T$ 
such that $T$ is rigid on $\rho_n$ and (power) rationally weakly mixing. 
In the definition of rational weak mixing, let $F=X_1 = X$ and assume without loss of generality 
that $\mu(F)=1$. Let $\mu_n$ be normalized Lebesgue probability measure on $X_n$. 
i.e. $\mu_n = {\mu} / {\mu(X_n)}$. 
Since each $R_n$ is weakly mixing and finite measure preserving on $X_n$, 
then for all $A,B\in P_n^{\prime}$, 
\[
\lim_{N\to \infty} \frac{1}{N}\sum_{i=0}^{N-1} | \mu_n(A\cap R_n^iB) - \mu_n(A)\mu_n(B) | = 0 . 
\]
If $u_i (n) = \mu(F\cap R_n^iF)$ and $a_N = \sum_{i=0}^{N-1} u_i (n)$, then 
for each $n\in \natural$, 
\[
\lim_{N\to \infty} \frac{a_N \mu(X_n)}{N} = 1 
\]
and 
\[
\lim_{N\to \infty} \frac{1}{a_N} \sum_{i=0}^{N-1} | u_i (n) - \frac{1}{\mu(X_n)} | = 0 . 
\]
This implies for all $A,B\in P_n^{\prime}$, 
\[
\lim_{N\to \infty} \frac{1}{a_N} \sum_{i=0}^{N-1} | \mu(A\cap R_n^iB) - \mu(A)\mu(B)u_i (n)| = 0 . 
\] 

Prior to establishing rational weak mixing, we prove 
a crucial lemma that was used in \cite{Towerplex1}. 
For $p\in P_n^{\prime}$,
$$
\frac{\mu (p)}{\mu (\tau_n (p))}=\frac{\mu (X_{n+1})}{\mu (X_n)}.
$$
It is straightforward to verify for any set $A$ measurable relative to $P_n^{\prime}$, 
\begin{align*}
\mu (A\triangle \tau_nA) &= \mu (A) - \mu (\tau_n(A)) \\ 
&\leq \mu (\tau_n A) [ \frac{\mu (X_{n+1})}{\mu (X_n)} - 1 ] 
= \frac{\mu (\tau_n A)}{\mu (X_n)} [ \mu (X_{n+1}) - \mu (X_n) ] . 
\end{align*}
and for any measurable set $C\subset X_n$, 
$$
| \mu (\tau_n^{-1}C) - \mu(C) | <  |\frac{\mu (X_{n+1})}{\mu (X_n)} - 1|.
$$
These two properties are used in the following lemma to show $R_{n+1}$ inherits dynamical 
properties from $R_n$ indefinitely over time. 
\begin{lemma}
\label{rescalinglem}
Suppose $\delta >0$ and $n\in \natural$ is chosen such that 
\begin{eqnarray*}
\epsilon_n + \mu (Y_n) < \frac{\delta}{6}.
\end{eqnarray*}
Then for $A,B\in Q_n$ and $i\in \natural$, the following holds:
\begin{enumerate}
\item $|\mu (R_{n+1}^iA\cap B)-\mu (A)\mu (B) u_i(n+1)| \\ 
< |\mu (R_{n}^iA\cap B)-\mu (A)\mu (B) u_i(n)| + [ {\delta} / {\mu (X_n)} ]$; 
\item $\mu (R_{n+1}^iA\triangle A) < \mu (R_{n}^iA\triangle A) + [ {\delta} / {2\mu (X_n)} ]$. 
\end{enumerate}
\end{lemma}
\begin{proof} 
For $A,B\in Q_n$, let $A^{\prime}= \tau_n^{-1} A$ and 
$B^{\prime}=\tau_n^{-1} B$. 
Thus, 
$\mu(A^{\prime}\triangle A) = \mu(\tau_n^{-1}(A\setminus \tau_nA)) < {\delta} / {6\mu (X_n)}$ 
and $\mu (B^{\prime}\triangle B)<\frac{\delta}{6\mu (X_n)}$. 
By applying the triangle inequality several times, we get the following approximation: 
\begin{eqnarray*} 
| \mu (R_{n+1}^iA\cap B) &-& \mu(R_n^i A \cap B) | \\ 
&\leq& | \mu (R_{n+1}^iA^{\prime}\cap B^{\prime}) - \mu(R_n^i A \cap B) | 
+ \frac{\delta}{3\mu (X_n)} \\ 
&=& \mu (\tau_n^{-1}R_{n}^i\tau_nA^{\prime}\cap B^{\prime}) - \mu(R_n^i A \cap B) | + \frac{\delta}{3\mu (X_n)} \\ 
&=& \mu (\tau_n^{-1}(R_{n}^i\tau_nA^{\prime} \cap \tau_n B^{\prime})) - \mu(R_n^i A \cap B) | + \frac{\delta}{3\mu (X_n)} \\ 
&=& | \mu( \tau_n^{-1}(R_n^i A \cap B)) - \mu(R_n^i A \cap B) | + \frac{\delta}{3\mu (X_n)} \\ 
&<& \frac{\delta}{2\mu (X_n)} . 
\end{eqnarray*}
Similarly, 
$$
| \mu (R_{n+1}^iF\cap F) - \mu(R_n^i F \cap F) | < \frac{\delta}{2\mu (X_n)} . 
$$
Hence, 
$$|\mu (R_{n+1}^iA\cap B)-\mu (A)\mu (B)u_i(n+1)| < |\mu (R_{n}^iA\cap B)-\mu (A)\mu (B)u_i(n)| + 
\frac{\delta}{\mu (X_n)}.$$
The second part of the lemma can be proven in a similar fashion using the triangle inequality. 
\begin{align*}
| \mu(R_{n+1}^iA\triangle A) & - \mu (R_{n}^iA\triangle A) | \leq
| \mu (R_{n+1}^iA^{\prime}\triangle A^{\prime}) - \mu (R_{n}^iA\triangle A) | + \frac{\delta}{3\mu (X_n)} \\ 
&= | \mu (\tau_n^{-1}R_n^i \tau_nA^{\prime}\triangle A^{\prime}) - \mu (R_{n}^iA\triangle A) | + \frac{\delta}{3\mu (X_n)} \\ 
&= | \mu (\tau_n^{-1} (R_n^i A \triangle A)) - \mu (R_n^i A \triangle A) | + \frac{\delta}{3\mu (X_n)} \\ 
&< \frac{\delta}{2\mu (X_n)} 
\end{align*}
Therefore, 
$$\mu(R_{n+1}^iA\triangle A) < \mu(R_{n}^iA\triangle A) + \frac{\delta}{2\mu (X_n)}$$ 
and our proof is complete. 
\end{proof}

\section{Approximation}
For probability preserving transformations, 
if the transformation is rigid on a dense collection 
of measurable sets, then the transformation is rigid on all measurable sets. 
Similarly, if a probability preserving transformation is mixing on a fixed sequence 
for all measurable sets from a dense collection, then the transformation is mixing 
on the same sequence. 
Since the normalizing term $a_N$ in the rationally weakly mixing condition may grow 
at a rate much slower than $N$, it is not clear this condition will hold for all 
measurable subsets $A\subseteq F$, when it holds for a dense collection of sets 
contained in $F$.  In this section, we give conditions that allow extension 
of the rational weak mixing condition from a dense collection of sets in $F$ 
to all measurable sets contained in $F$. Let $P$ be a dense collection of sets, 
each a subset of $F$. 

\begin{lemma}
\label{approx1}
Suppose there exist a sequence of measurable sets 
$F=X_1 \subset X_2 \subset \ldots$, and a sequence of natural numbers 
$M_1 < M_2 < \ldots$ such that for each $A\in P$ 
and any sequence $N_n$ satisfying $M_n \leq N_n < M_{n+1}$, 
\begin{equation}
\lim_{n \to \infty} \frac{\mu (X_n)^{2}}{N_n} \sum_{i=0}^{N_n-1} 
| \mu (A\cap T^iA) - \mu (A)^2 \frac{1}{\mu (X_n)} | = 0 . \label{frwm}
\end{equation}
Then for any measurable set $E\subseteq F$ and $A\in P$, 
\[
\lim_{n \to \infty} \frac{\mu (X_n)}{N_n} \sum_{i=0}^{N_n-1} 
| \mu (E\cap T^iA) - \mu (E) \mu (A) \frac{1}{\mu (X_n)} | = 0 . 
\]
\end{lemma}
\begin{proof}
If not true, then there exists $\delta > 0$ and $\ell \in \natural$ 
such that for $n \geq \ell$, 
\[
\frac{\mu (X_n)}{N_n} \sum_{i=0}^{N_n-1} 
| \mu(E\cap T^iA) - \mu(E)\mu(A)\frac{1}{\mu(X_n)} | > 2\delta . 
\]
There exists $\Gamma_n \subseteq \{0,1,\ldots ,N_n-1\}$ such that 
$| \Gamma_n | \geq \delta N_n$, 

\noindent 
$\mu (E\cap T^iA) - {\mu(E)\mu(A)}/{\mu(X_n)} \geq 0$ (or $\leq 0$) 
for $i\in \Gamma_n$ and 
\[
\frac{\mu (X_n)}{|\Gamma_n|} 
\sum_{i\in \Gamma_n} (\mu(E\cap T^iA) - \mu(E)\mu(A)\frac{1}{\mu(X_n)}) > \delta .
\]
On the other hand, we can use the Cauchy-Schwarz inequality to obtain,
\begin{align}
&\frac{\mu(X_n)}{|\Gamma_n|} \sum_{i\in \Gamma_n} 
(\mu(E\cap T^iA) - \mu(E)\mu(A)\frac{1}{\mu(X_n)}) \\ 
&\leq \mu(X_n) \int_{X_n} 
(\frac{1}{|\Gamma_n|}\sum_{i\in \Gamma_n} 
I_{T^iA}(x) - \frac{\mu(A)}{\mu(X_n)}) I_E(x) d\mu \\ 
&\leq \mu(X_n)[\int_{X_n} 
(\frac{1}{|\Gamma_n|}\sum_{i\in \Gamma_n} 
I_{T^iA}(x) - \frac{\mu(A)}{\mu(X_n)})^2 d\mu]^{\frac{1}{2}} 
[\int_{X_n} I_E(x)d\mu]^{\frac{1}{2}} \\ 
&\leq [ 
\frac{\mu(X_n)^2}{|\Gamma_n|^2} \sum_{i,j\in \Gamma_n} 
|\mu(T^iA\cap T^jA) - \frac{\mu(A)^2}{\mu(X_n)} | ]^{\frac{1}{2}} 
\sqrt{\mu(E)} . \label{frwm2}
\end{align}
However, condition (\ref{frwm}) implies that expression 
(\ref{frwm2}) converges to zero.
\end{proof}
The following lemma uses Lemma \ref{approx1} to extend 
the rational weak mixing condition to all measurable 
subsets of $F$. 
\begin{lemma}
\label{approx2}
If $T$ is conservative ergodic and 
satisfies the same conditions of Lemma \ref{approx1}, 
then $T$ is rationally weakly mixing. 
In particular, for any measurable sets $D,E\subseteq F$ 
and $M_n \leq N_n < M_{n+1}$, 
\[
\lim_{n \to \infty} \frac{\mu (X_n)}{N_n} \sum_{i=0}^{N_n-1} 
| \mu (E\cap T^iD) - \mu (E) \mu (D) \frac{1}{\mu (X_n)} | = 0 . 
\]
\end{lemma}
\begin{proof}
Let $D$ and $E$ be measurable subsets of $F$, 
and let $\eta > 0$. 
Choose $A,B \in P$ such that 
$\mu (A\triangle D) < \eta$ and $\mu (B\triangle E) < \eta$. 
Without loss of generality, let 
$D=A\cap D$ and $E=B\cap E$. 
A straightforward application of the triangle inequality 
gives the following bounds,
\begin{enumerate}
\item $\frac{\mu(X_n)}{N_n} \sum_{i=0}^{N_n-1} 
| \mu(A)\mu(B)\frac{1}{\mu(X_n)} - \mu(A)\mu(E)\frac{1}{\mu(X_n)} | 
< \eta$,
\item $\frac{\mu(X_n)}{N_n} \sum_{i=0}^{N_n-1} 
| \mu(A)\mu(B)\frac{1}{\mu(X_n)} - \mu(D)\mu(B)\frac{1}{\mu(X_n)} | 
< \eta$,
\item $\frac{\mu(X_n)}{N_n} \sum_{i=0}^{N_n-1} 
| \mu(A)\mu(B)\frac{1}{\mu(X_n)} - \mu(D)\mu(E)\frac{1}{\mu(X_n)} | 
< 2\eta$. 
\end{enumerate}
From Lemma \ref{approx1}, we have that 
\begin{align}
\frac{\mu(X_n)}{N_n} & \sum_{i=0}^{N_n-1} 
| \mu(D\cap T^iB) - \mu(A\cap T^iB) | \\ 
& \leq \frac{\mu(X_n)}{N_n} \sum_{i=0}^{N_n-1} 
| \mu(D\cap T^iB) - \mu(D)\mu(B)\frac{1}{\mu(X_n)} | \\ 
& + \frac{\mu(X_n)}{N_n}  \sum_{i=0}^{N_n-1} 
| \mu(A\cap T^iB) - \mu(A)\mu(B)\frac{1}{\mu(X_n)} | \\ 
& + \frac{\mu(X_n)}{N_n} \sum_{i=0}^{N_n-1} 
| \mu(A)\mu(B)\frac{1}{\mu(X_n)} - \mu(D)\mu(B)\frac{1}{\mu(X_n)} | \\ 
&\longrightarrow \eta_1 \leq \eta 
\end{align}
for some real number $\eta_1 \geq 0$. 
Similarly, 
\[
\frac{\mu(X_n)}{N_n} \sum_{i=0}^{N_n-1} 
| \mu(A\cap T^iE) - \mu(A\cap T^iB) | \to \eta_2 
\]
for some nonnegative real number $\eta_2 \leq \eta$. 
Finally, we have 
\begin{align}
\frac{\mu(X_n)}{N_n} & \sum_{i=0}^{N_n-1} 
| \mu(D\cap T^iE) - \mu(D)\mu(E)\frac{1}{\mu(X_n)} | \\ 
&\leq \frac{\mu(X_n)}{N_n} \sum_{i=0}^{N_n-1} 
| \mu(A\cap T^iB) - \mu(A)\mu(B)\frac{1}{\mu(X_n)} | \\ 
&+ \frac{\mu(X_n)}{N_n} \sum_{i=0}^{N_n-1} 
| \mu(A)\mu(B)\frac{1}{\mu(X_n)} - \mu(D)\mu(E)\frac{1}{\mu(X_n)} | \\ 
&+ \frac{\mu(X_n)}{N_n} \sum_{i=0}^{N_n-1} \Big( 
\mu((A\setminus D)\cap T^iB) 
+ \mu(A\cap T^i(B\setminus E)) \Big) \\ 
&\longrightarrow \eta_3 \leq 4\eta . 
\end{align}
\end{proof}

\section{Rational Weak Mixing and Rigid} 
\label{main1}
To ensure conservativity and ergodicity, 
the same technique from \cite{Towerplex1} 
may be used, or directly modify the choice 
of $M_n$, $\epsilon_n$ and $h_n$ below, to force $F=X_1$ to sweep out. 
Suppose $\delta_n > 0$ such that $\lim_{n\to \infty} \delta_n = 0$. 
Fix $n\in \natural$. 
Suppose $M_{n-1}$, $h_{n-1}$ and $\epsilon_{n-1}$ have been chosen. 
Choose $M_n > \max{ \{ h_{n-1}, M_{n-1} \} }$ such that 
for all $A\in P_n^{\prime}$ and $N \geq M_n$, 
\begin{eqnarray}
\frac{\mu(X_n)^2}{N} \sum_{i=0}^{N-1} 
| \mu(A\cap R_n^iA) - \mu(A)^2\frac{1}{\mu(X_n)} | < \delta_n . \label{cond1}
\end{eqnarray} 
Choose $\epsilon_n > 0$ and $h_n > M_n$ such that 
\[
\epsilon_n n M_n < \epsilon_{n-1} \ \ \ \mbox{and}\ \ \ 
\frac{1}{h_n} n M_n < \epsilon_{n-1} . 
\]
\begin{proof}[Proof of rational weak mixing] 
Fix $k\in \natural$ and $A\in P_k^{\prime}$. 
Suppose $N_n \in \natural$ such that $M_n \leq N_n < M_{n+1}$. 
By using the first approximation from Lemma \ref{rescalinglem}, 
\[
\lim_{n\to \infty} \frac{\mu(X_{n+1})^2}{N_n} \sum_{i=0}^{{N_n} - 1} | \mu(A\cap R_{n+1}^iA) 
- \mu(A)^2\frac{1}{\mu(X_{n+1})} | = 0 . 
\] 
Set $E_{n+1}=\{x\in X_n: T_{n+2}(x)\neq T_{n+1}(x)\}$. Let 
$$E_{n+1}^{\prime}=\bigcup_{i=0}^{M_{n+1}-1}[T_{n+2}^{-i}E_{n+1}\cup T_{n+1}^{-i}E_{n+1}]$$
Thus, $\mu (E_{n+1}^{\prime}) < 2M_{n+1} \kappa \epsilon_{n+1}$. 
For $x\notin E_{n+1}^{\prime}$, $T_{n+2}^i(x)=T_{n+1}^i(x)$ for $0\leq i\leq M_{n+1}$. 
Let $E_{n+1}^{\prime\prime} = \bigcup_{k=n+1}^{\infty}E_k^{\prime}$. 
For $x\notin E_{n+1}^{\prime\prime}$ and $0\leq i\leq M_{n+1}$, 
$T^i(x) = T_{n+1}^i(x)$. 
Also,
$$\mu (E_{n+1}^{\prime\prime} ) < \sum_{k=n+1}^{\infty} 2M_{k}\kappa \epsilon_k 
< \frac{1}{n+1}\sum_{k=n+1}^{\infty}2\kappa \epsilon_{k-1}$$
and $\sum_{k=n+1}^{\infty}2\kappa \epsilon_{k-1} \to 0$ as $n\to \infty$. 
Hence, 
\begin{eqnarray*}
\lim_{n\to \infty} \frac{\mu(X_n)^2}{N_n}\sum_{i=0}^{{N_n}-1} 
| \mu (T^iA\cap A) - \mu (A)^2 \frac{1}{\mu(X_n)} | = 0 . 
\end{eqnarray*}
Therefore, by Lemma \ref{approx2}, 
$T$ is rationally weakly mixing. 
\end{proof} 
Rigidity on $\rho_n$ can be established in a similar fashion, 
using approximation (2) from Lemma \ref{rescalinglem}, 
and similar choices for $M_n$, $\epsilon_n$ and $h_n$. 

\section{Power Rational Weak Mixing}
We show that the techniques given in this paper can be applied 
to the class of power rational weak mixing transformations. 
We can use the constructions defined previously in this paper. 
We need to update the choice of the parameters 
$M_n$, $\epsilon_n$ and $h_n$. 
Let $V$ be the collection of all finite vectors 
comprised of nonzero integers.  The collection $V$ 
is countable, so we can order 
$V=\{ v_1, v_2, \ldots \}$. 
For each $v\in V$ and $n\in \natural$, 
define the finite measure preserving transformation 
\[
\mathcal{R}_{n,v} = R_n^{u_1} \times R_n^{u_2} \times 
\ldots \times R_n^{u_{|v|}}
\]
where $v = \langle u_1,u_2,\ldots ,u_{|v|} \rangle$. 
Products of sets from $P_n^{\prime}$ may be used 
to produce a finite approximating collection 
for the Cartesian product space. 
Also, Lemma \ref{rescalinglem} may be extended 
in a straightforward manner to subsets 
of the product space. Note the map 
$\tau_n$ can be applied pointwise to produce 
an analogous isomorphism on products. 
Suppose $j_n$ is a sequence of natural numbers 
such that 
$\lim_{n\to \infty} j_n = \infty$. 
Now, replace condition (\ref{cond1}) above 
with the following condition, 
\[
\frac{\mu(X_n)^2}{N} \sum_{i=0}^{N-1} 
| \mu(A\cap \mathcal{R}_{n,v_j}^iA) - \mu(A)^2\frac{1}{\mu(X_n)} | 
< \delta_n 
\]
and require this hold for $1\leq j\leq j_n$ and $N \geq M_n$. 
This is possible, since $R_n$ is finite measure preserving, 
weak mixing, and all finite Cartesian products 
of nonzero powers of $R_n$ will be weak mixing. 
In a manner similar to the case of a single transformation $T$, 
we can force the product transformation to be 
conservative ergodic by ensuring the set 
$X_1\times X_1\times \ldots X_1$ sweeps out 
under the product transformation. 
The rest of the arguments from the proof of Theorem \ref{infinite_towerplex} 
go through in the same manner, but with 
$\mathcal{R}_{n,v_j}$ replacing $R_n$ and 
\[
T^{u_1}\times T^{u_2}\times \ldots \times T^{u_{|v_j|}}
\]
replacing $T$.  If the sequence $j_n$ grows slowly enough, 
then we still have 
$\sum_{k=n+1}^{\infty}2\kappa \epsilon_{k-1} \to 0$ as $n\to \infty$, 
and our result follows. $\Box$

The corollaries below follow from Theorem \ref{infinite_towerplex2} 
and corollaries given in \cite{Towerplex1}. 
Given a sequence $\mathcal{A}$, 
define the density function $g_{\mathcal{A}}:\natural \to [0,1]$ 
such that $g_{\mathcal{A}}(k)={\#(\mathcal{A} \cap \{1,2,\ldots k\})} / {k}$. 
\begin{corollary}
\label{slowrigid}
Given any real-valued function $f:\natural \to (0,\infty)$ such that 
$$\lim_{n\to \infty}f(n)=0,$$ 
there exists an infinite measure preserving, 
power rationally weakly mixing transformation 
with rigidity sequence $\mathcal{A}$ such that 
$$
\lim_{n\to \infty}\frac{f(n)}{g_{\mathcal{A}}(n)}=0.
$$
Also, there exist infinite measure preserving, 
power rationally weakly mixing transformations 
with rigidity sequences 
$\rho_n$ satisfying 
$$
\lim_{n\to \infty}\frac{\rho_{n+1}}{\rho_n}=1.
$$
\end{corollary}

\begin{corollary}
\label{rotations}
Let $\alpha\in (0,1)$ be any irrational number, and let 
$\rho_n$ be a sequence of natural numbers satisfying 
$$
\lim_{n\to \infty} |\exp{(2\pi i\alpha \rho_n)}-1|=0.
$$
Then there exists an infinite measure preserving, 
power rationally weakly mixing transformation 
$T$ such that $\rho_n$ is a rigidity sequence for $T$. 
\end{corollary}

\subsection*{Acknowledgements}
The author wishes to thank Jon Aaronson, Nathaniel Friedman, 
Mariusz Lema\'nczyk, Cesar Silva and several others 
for valuable feedback on this research.

\end{document}